\documentclass[10pt]{article}
\usepackage{amssymb}
\newtheorem{precor}{{\bf Corollary}}

\newenvironment{cor}{\begin{precor}{\hspace{-0.5
               em}{\bf.\ }}}{\end{precor}}
\newtheorem{precon}{{\bf Conjecture}}

\newtheorem{prealphcon}{{\bf Conjecture}}

\newtheorem{predefin}{{\bf Definition}}

\newenvironment{defin}[1]{\begin{predefin}{\hspace{-0.5
                   em}{\bf.\ }}{\rm #1}\hfill{$\spadesuit$}}{\end{predefin}}
\newtheorem{preexm}{{\bf Example}}

\newtheorem{preappl}{{\bf Application}}

\newtheorem{prelem}{{\bf Lemma}}

\newtheorem{preproof}{{\bf Proof.\ }}

\newenvironment{proof}[1]{\begin{preproof}{\rm
               #1}\hfill{$\blacksquare$}}{\end{preproof}}
\newtheorem{pretheorem}{{\bf Theorem}}

\newenvironment{theorem}{\begin{pretheorem}{\hspace{-0.5
               em}{\bf.\ }}}{\end{pretheorem}}
\newtheorem{prealphtheorem}{{\bf Theorem}}

\newtheorem{prealphlem}{{\bf Lemma}}

\newtheorem{prepro}{{\bf Proposition}}

\newtheorem{preprb}{{\bf Problem}}

\newtheorem{prerem}{{\bf Remark}}

\newtheorem{preapp}{{\bf Application}}

\newtheorem{prequ}{{\bf Question}}

\def\conct[#1,#2]{\mbox {${#1} \leftrightarrow {#2}$}}
\def\dconct[#1,#2]{\mbox {${#1} \rightarrow {#2}$}}
\def\deg[#1,#2]{\mbox {$d_{_{#1}}(#2)$}}
\def\mindeg[#1]{\mbox {$\delta_{_{#1}}$}}
\def\maxdeg[#1]{\mbox {$\Delta_{_{#1}}$}}
\def\outdeg[#1,#2]{\mbox {$d_{_{#1}}^{^+}(#2)$}}
\def\minoutdeg[#1]{\mbox {$\delta_{_{#1}}^{^+}$}}
\def\maxoutdeg[#1]{\mbox {$\Delta_{_{#1}}^{^+}$}}
\def\indeg[#1,#2]{\mbox {$d_{_{#1}}^{^-}(#2)$}}
\def\minindeg[#1]{\mbox {$\delta_{_{#1}}^{^-}$}}
\def\maxindeg[#1]{\mbox {$\Delta_{_{#1}}^{^-}$}}

\def\dre[#1,#2,#3]{\mbox {${\cal E}^{^{#3}}(#1,#2)$}}
\def\var[#1,#2]{\mbox {${\rm Var}_{_{#1}}(#2)$}}
\def\ls[#1]{\mbox {$\xi^{^{#1}}$}}
\def\hom[#1,#2]{\mbox {${\rm Hom}({#1},{#2})$}}
\def\onvhom[#1,#2]{\mbox {${\rm Hom^{v}}(#1,#2)$}}
\def\onehom[#1,#2]{\mbox {${\rm Hom^{e}}(#1,#2)$}}
\def\core[#1]{\mbox {$#1^{^{\bullet}}$}}
\def\cay[#1,#2]{\mbox {${\rm Cay}({#1},{#2})$}}
\def\sch[#1,#2,#3]{\mbox {${\rm Sch}({#1},{#2},{#3})$}}
\def\cays[#1,#2]{\mbox {${\rm Cay_{s}}({#1},{#2})$}}
\def\dirc[#1]{\mbox {$\stackrel{\rightarrow}{C}_{_{#1}}$}}
\def\cycl[#1]{\mbox {${\bf Z}_{_{#1}}$}}

\begin{document}

\begin{center}
{\Large \bf On $b$-continuity of Kneser Graphs of Type $KG(2k+1,k)$}\\
\vspace{0.3 cm}
{\bf Saeed Shaebani}\\
{\it Department of Mathematical Sciences}\\
{\it Institute for Advanced Studies in Basic Sciences {\rm (}IASBS{\rm )}}\\
{\it P.O. Box {\rm 45195-1159}, Zanjan, Iran}\\
{\tt s\_shaebani@iasbs.ac.ir}\\ \ \\
\end{center}
\begin{abstract}
\noindent In this paper, we will introduce an special kind of graph homomorphisms namely semi-locally-surjective graph homomorphisms and show some relations between semi-locally-surjective graph homomorphisms and colorful colorings of graphs and then we prove that for each natural number $k$, the Kneser
graph $KG(2k+1,k)$ is $b$-continuous. Finally, we introduce some
special conditions for graphs to be $b$-continuous.

\end{abstract}
\section{Introduction}
All graphs considered in this paper are finite and simple
(undirected, loopless and without multiple edges). Let $G=(V,E)$
be a graph and $k\in \mathbb{N}$ and let $[k]:=\{i|\ i\in
\mathbb{N},\ 1\leq i\leq k \}$. A $k$-coloring (proper k-coloring)
of $G$ is a function $f:V\rightarrow [k]$ such that for each
$1\leq i\leq k$, $f^{-1}(i)$ is an independent set. We say that
$G$ is $k$-colorable whenever $G$ has a $k$-coloring $f$, in this
case, we denote $f^{-1}(i)$ by $V_{i}$ and call each $1\leq i\leq
k$, a color (of $f$) and each $V_{i}$, a color class (of $f$). The
minimum integer $k$ for which $G$ has a $k$-coloring, is called
the chromatic number of G and is denoted by $\chi(G)$.

Let $G$ be a graph and $f$ be a $k$-coloring of $G$ and $v$ be a
vertex of $G$. The vertex $v$ is called $b$-dominating ( or
colorful or color-dominating) ( with respect to $f$) if each
color $1\leq i\leq k$ appears on the closed neighborhood of $v$ (
$f(N[v])=[k]$ ). The coloring $f$ is said to be a colorful
$k$-coloring of $G$ if each color class $V_{i}\ (1\leq i\leq k) $
contains a $b$-dominating vertex $x_{i}$. Obviously, every
$\chi(G)$-coloring of $G$ is a colorful $\chi(G)$-coloring of
$G$. We denote ${\rm B}(G)$ the set of all positive integers $k$
for which $G$ has a colorful $k$-coloring. The maximum of ${\rm
B}(G)$, is called the $b$-chromatic number of $G$ and is denoted
by $b(G)$ (or $\phi(G)$ or $\chi_{b}(G)$). The graph $G$ is said
to be $b$-continuous if each integer $k$ between $\chi(G)$ and
$b(G)$ is an element of ${\rm B}(G)$. There are graphs that are
not $b$-continuous, for example, the 3-dimensional cube $Q_{3}$
is not $b$-continuous, because $2\in {\rm B}(G)$ and $4\in {\rm
B}(G)$ but $3\notin {\rm B}(G)$ ( \cite{irv} ). We have to note that
the problem of deciding whether graph $G$ is $b$-continuous is
NP-complete ( \cite{bar} ). The colorful coloring of graphs was
introduced in 1999 in \cite{irv} with the terminology
$b$-coloring.

Let $m,n\in \mathbb{N}$ and $m\leq n$. $KG(n,m)$ is the graph
whose vertex set is the set of all subsets of size $m$ of $[n]$ in
which two vertices $X$ and $Y$ are adjacent iff $X\bigcap
Y=\emptyset$. Note that $KG(5,2)$ is the famous Petersen graph.
It was conjectured by Kneser in 1955 ( \cite{KG} ), and proved by
Lov{\'a}sz in 1978 ( \cite{lov} ), that if $n\geq 2m$, then $\chi
(KG(n,m))=n-2m+2$. Lov{\'a}sz's proof was the beginning of using
algebraic topology in combinatorics. Colorful colorings of Kneser
graphs have been investigated in \cite{haj} and \cite{jav}.
Javadi and Omoomi in \cite{jav} showed that for $n\geq 17$,
$KG(n,2)$ is $b$-continuous. Only a few classes of graphs are
known to be $b$-continuous (see \cite{bar}, \cite{fai} and
\cite{jav}). We want to prove that for each natural number $k$,
$KG(2k+1,k)$ is $b$-continuous. In this regard, first we
introduce an special kind of graph homomorphisms which is related
to colorful colorings of graphs.

\begin{defin}{ Let $G$ and $H$ be graphs. A function
$f:V(G)\rightarrow V(H)$ is called a semi-locally-surjective graph homomorphism
from $G$ to $H$ if $f$ is a surjective graph homomorphism
from $G$ to $H$ and satisfies the following condition :
\\
\\
 $\forall u \in V(H):\  \exists a\in f^{-1}(u)\  s.t\  \forall v\in N_{H}(u):\  \exists b\in f^{-1}(v)\  s.t\  \{a,b\}\in E(G) $.

}\end{defin}

We know that a graph $G$ is $k$-colorable iff there exists a graph homomorphism from $G$ to the complete graph $K_{k}$ and the chromatic number of $G$ is the least natural number $k$ for which there exists a graph homomorphism from $G$ to $K_{k}$. Indeed, we can think of graph homomorphisms from graphs to complete graphs instead of graph colorings. The following obvious theorem shows such a similar relation between colorful colorings of graphs and semi-locally-surjective graph homomorphisms. Indeed, we can think of semi-locally-surjective graph homomorphisms from graphs to complete graphs instead of colorful colorings of graphs.

\begin{theorem}{ Let $G$ be a graph and $k\in \mathbb{N}$. Then $k\in {\rm B}(G)$ iff there
exists a semi-locally-surjective graph homomorphism from $G$ to
$K_{k}$. Also, the chromatic number of $G$ ( $\chi(G)$ ) and the $b$-chromatic number of $G$ ( $b(G)$ ) are respectively the least and the greatest natural numbers $k$ for which there exists a semi-locally-surjective graph homomorphism from $G$ to $K_{k}$  . }
\end{theorem}

We know that the composition of two graph homomorphisms is again a graph homomorphism. A similar theorem holds for composition of semi-locally-surjective graph homomorphisms.

\begin{theorem}{ Let $G_{1}$, $G_{2}$ and $G_{3}$ be graphs. If $g$ is a semi-locally-surjective graph homomorphism from $G_{2}$ to $G_{1}$ and $f$ is a semi-locally-surjective graph homomorphism from $G_{3}$ to
$G_{2}$, then $gof$ is a semi-locally-surjective graph homomorphism from $G_{3}$ to
$G_{1}$. }
\end{theorem}

The following theorem shows another relation between semi-locally-surjective graph homomorphisms and colorful colorings of graphs.

\begin{theorem}{\label{thm1} Let $G_{1}$ and $G_{2}$ be graphs. If there
exists a semi-locally-surjective graph homomorphism from $G_{1}$ to
$G_{2}$, then ${\rm B}(G_{2}) \subseteq {\rm B}(G_{1})$. }
\end{theorem}
\begin{proof}{ Let $f$ be a semi-locally-surjective graph homomorphism
from $G_{1}$ to $G_{2}$, $k\in {\rm B}(G_{2})$, and $V_{1},\ldots ,V_{k}$
be color classes of a colorful $k$-coloring of $G_{2}$ and $x_{1},\ldots ,x_{k}$
be some $b$-dominating vertices of $G_{2}$ with respect to this $k$-coloring and $x_{i}\in V_{i}\ (1\leq i\leq k)$. Obviously, $f^{-1}(V_{1}),\ldots
,f^{-1}(V_{k})$ are nonempty color classes of a $k$-coloring of $G_{1}$ and $f^{-1}(x_{1}),\ldots ,f^{-1}(x_{k})$
are some $b$-dominating vertices of $G_{1}$ with respect to this $k$-coloring and $f^{-1}(x_{i})\in f^{-1}(V_{i})\ (1\leq i\leq k)$. Therefore, $G_{1}$ has a colorful $k$-coloring and $k\in {\rm B}(G_{1})$. Hence, ${\rm B}(G_{2})
\subseteq {\rm B}(G_{1})$. }
\end{proof}

Now we prove that for each natural number $k$, $KG(2k+1,k)$ is
$b$-continuous.

\begin{theorem}{ For each $k\in \mathbb{N}$, $KG(2k+1,k)$ is
$b$-continuous. }
\end{theorem}
\begin{proof}{
For each $k\in \mathbb{N}$, $\chi (KG(2k+1,k))=3$. Note that
${\rm B}(KG(3,1))={\rm B}(K_{3})=\{3\}$ and therefore, for $k=1$
the assertion follows. Blidia, et al. in \cite{bli} proved that
the $b$-chromatic number of the Petersen graph is 3 and therefore,
${\rm B}(KG(5,2))={\rm B}(Petersen\ graph)=\{3\}$. Hence,
$KG(2k+1,k)$ is $b$-continuous for $k=2$. For $k\geq3$, the
function $f:V(KG(2k+3,k+1))\rightarrow V(KG(2k+1,k))$ which
assigns to each $A\subseteq [2k+3]$ with $|A\bigcap
\{2k+2,2k+3\}|\leq1$, $f(A)=A\setminus\{\max A\}$ and to each
$A\subseteq [2k+3]$ with $\{2k+2,2k+3\}\subseteq A$,
$f(A)=(A\setminus\{2k+2,2k+3\})\bigcup \{\max ([2k+1]\setminus
A)\}$, is a surjective graph homomorphism from $KG(2k+3,k+1)$ to
$KG(2k+1,k)$. Now for each $X\in V(KG(2k+1,k))$, $(X \bigcup \{2k+2\})\in f^{-1}(X)$ and for each $Y\in N_{KG(2k+1,k)}(X)$, $(Y \bigcup \{2k+3\})\in f^{-1}(Y)$ and $\{X \bigcup \{2k+2\},Y \bigcup \{2k+3\}\}\in E(KG(2k+3,k+1))$. Hence, $f$ is a semi-locally-surjective graph homomorphism from $KG(2k+3,k+1)$ to
$KG(2k+1,k)$. Consequently, Theorem \ref{thm1} implies that
${\rm B}(KG(2k+1,k)) \subseteq {\rm B}(KG(2k+3,k+1))$, besides,
${\rm B}(KG(7,3))\subseteq {\rm B}(KG(9,4))\subseteq ...\subseteq
{\rm B}(KG(2n+1,n)) \subseteq ...$ . (I)

On the other hand, Javadi and Omoomi in \cite{jav} showed that
for $k\geq 3$, $b(KG(2k+1,k))=k+2$ and $k+2\in {\rm
B}(KG(2k+1,k))$. Therefore, for each $k\geq 3$, $\{i+2|\
i\in\mathbb{N},\ 3\leq i\leq k\}\subseteq {\rm B}(KG(2k+1,k))$.
Also, since $\chi (KG(2k+1,k))=3$, $3\in {\rm B}(KG(2k+1,k))$.
So, constructing a colorful 4-coloring of $KG(2k+1,k)\ (k\geq3)$
completes the proof. (I) implies that it is enough to construct a
colorful 4-coloring of $KG(7,3)$. Set
$$\begin{array}{cc}
V_{1}:=\{\ \{1 ,2 ,3 \},\{1 ,4 ,5 \},\{2 , 5,6 \},\{ 1,2 ,6
\},\{ 1,2 ,7 \},\{ 1,3 ,6 \},\{ 1,6 ,7 \},  &  \\
\{ 1, 4, 6\}\ \},   &  \\
V_{2}:=\{\ \{5 , x, y\}\ |\ x,y\in \{1,2,3,4,6,7\},\ x\neq y\
\}\setminus\{\ \{1 , 4,5 \},\{2 ,5 ,6 \},   &  \\
\{4 ,5 , 7\}\ \},   &  \\
V_{3}:=\{\ \{1 ,2 ,4 \},\{1 ,3 , 7\},\{4 ,5 ,7 \},\{1 ,4 ,7
\},\{2 ,6 ,7 \}\ \},   & \\
V_{4}:=(\{\ \{4 ,x ,y \}\ |\ x,y\in \{1,2,3,6,7\},\ x\neq y
\}\setminus\{\ \{1 ,2 ,4 \},\{ 1,4 ,6 \}, & \\
\{1 ,4 ,7 \}\ \})\ \bigcup\ \{\ \{2 ,3 ,6 \},\{2 ,3 ,7 \},\{3 ,6
,7 \}\ \}. &
\end{array}$$

Now, one can check that $V_{1},\ V_{2},\ V_{3},\ V_{4}$ are color
classes of a colorful 4-coloring of $KG(7,3)$ that $ \{1 ,2 ,3
\}\in V_{1},\ \{5 ,6 ,7 \}\in V_{2},\ \{2 ,6 ,7 \}\in V_{3}$ and
$\{1 ,3 ,4 \}\in V_{4}$ are some $b$-dominating vertices with
respect to this $4$-coloring.}
\end{proof}

The semi-locally-surjective graph homomorphism $f$ in above Theorem can
be generalized as follows.

\begin{theorem}{ Let $n,m\in\mathbb{N}$ with $n>2m$.
Then ${\rm B}(KG(n,m))\subseteq {\rm B}(KG(n+2,m+1))$.}
\end{theorem}
\begin{proof}
{The function $f:V(KG(n+2,m+1))\rightarrow V(KG(n,m))$ which
assigns to each $A\subseteq [n+2]$ with $|A\bigcap
\{n+1,n+2\}|\leq1$, $f(A)=A\setminus\{\max A\}$ and to each
$A\subseteq [n+2]$ with $\{n+1,n+2\}\subseteq A$,
$f(A)=(A\setminus\{n+1,n+2\})\bigcup \{\max ([n]\setminus A)\}$,
is a surjective graph homomorphism from $KG(n+2,m+1)$ to $KG(n,m)$. Now for each $X\in V(KG(n,m))$, $(X \bigcup \{n+1\})\in f^{-1}(X)$ and for each $Y\in N_{KG(n,m)}(X)$, $(Y \bigcup \{n+2\})\in f^{-1}(Y)$ and $\{X \bigcup \{n+1\},Y \bigcup \{n+2\}\}\in E(KG(n+2,m+1))$. Hence, $f$ is a semi-locally-surjective graph homomorphism from $KG(n+2,m+1)$ to
$KG(n,m)$ and therefore, Theorem \ref{thm1} implies that
${\rm B}(KG(n,m)) \subseteq {\rm B}(KG(n+2,m+1))$.}
\end{proof}

\begin{cor}{ Let $a,b\in\mathbb{N}\bigcup\{0\}$ and $a>2b$. Also, for each $i\in \mathbb{N}\setminus\{1\}$, let ${\rm B}_{i}:={\rm B}(KG(2i+a,i+b))$ and ${\rm b_{i}}:=b(KG(2i+a,i+b))$.
Then ${\rm B}_{2}\subseteq {\rm B}_{3}\subseteq {\rm
B}_{4}\subseteq ...\subseteq {\rm B}_{n}\subseteq {\rm
B}_{n+1}\subseteq ...$ , and ${\rm b_{2}}\leq {\rm b_{3}}\leq {\rm
b_{4}}\leq ...\leq {\rm b_{n}}\leq {\rm b_{n+1}}\leq ...$} .
\end{cor}

Now we introduce some special conditions for graphs to be
$b$-continuous. But first we note that in a graph $G$ with at
least one cycle, the girth of $G$ ($g(G)$), is the minimum of all
cycle lengths of $G$ and if $G$ has not any cycles, the girth of
$G$ is defined $g(G)=+\infty$ .
\\

Blidia, et al. proved the following theorem.
\begin{theorem}{( {\rm \cite{bli}} )\label{thm5} If $d\leq6$, then for every $d$-regular graph $G$ with
girth $g(G)\geq5$ which is different from the Petersen graph,
$b(G)=d+1$. }
\end{theorem}

By using this theorem, we prove the following theorem.

\begin{theorem}{ Let $3\leq d\leq6$ and for each $2\leq i\leq d$,
$G_{i}$ be an $i$-regular graph with girth $g(G_{i})\geq5$ which
 is different from the Petersen graph. Also, suppose that for each $3\leq i\leq d$, there exists
 a semi-locally-surjective graph homomorphism $f_{i}$ from $G_{i}$ to $G_{i-1}$.
Then for each $2\leq i\leq d$, $G_{i}$ is $b$-continuous. }
\end{theorem}

\begin{proof}{
Theorem \ref{thm5} implies that for each $2\leq i\leq d$,
$b(G_{i})=i+1$ and therefore, $i+1 \in {\rm B}(G_{i}) $. Also,
since for each $3\leq i\leq d$, there exists a
semi-locally-surjective graph homomorphism $f_{i}$ from $G_{i}$
to $G_{i-1}$, theorem \ref{thm1} implies that ${\rm
B}(G_{i-1})\subseteq {\rm B}(G_{i})$ and consequently, ${\rm
B}(G_{2})\subseteq {\rm B}(G_{3})\subseteq  ... \subseteq {\rm
B}(G_{d}).$ Hence, for each $2\leq i\leq d$, $\{j+1|2\leq j\leq
i\}\subseteq {\rm B}(G_{i})$ and therefore,
$\{3,4,...,i+1\}\subseteq {\rm B}(G_{i})$. Now, there are 2 cases:

Case 1) The case that $G_{i}$ is bipartite. In this case,
$\chi(G_{i})=2$ and therefore, $2\in {\rm B}(G_{i})$ and
$\{2,3,...,i+1\}\subseteq {\rm B}(G_{i})$, so ${\rm
B}(G_{i})=\{2,3,...,i+1\}$ and $G_{i}$ is $b$-continuous.

Case 2) The case that $G_{i}$ is not bipartite. In this case,
$\chi(G_{i})\geq 3$ and since $\{3,...,i+1\}\subseteq B(G_{i})$,
so ${\rm B}(G_{i})=\{3,...,i+1\}$ and $G_{i}$ is $b$-continuous.
\\
\\
Therefore, for each $2\leq i\leq d$, $G_{i}$ is $b$-continuous. }
\end{proof}

\end{document}